\documentclass[11pt,a4paper]{article}
\usepackage[a4paper]{geometry}
\usepackage{amssymb,latexsym,amsmath,amsfonts,amsthm}
\usepackage{graphicx}

\usepackage{epsfig}
\usepackage{comment}
\usepackage{overpic}
\usepackage{mathrsfs,enumerate}
\usepackage{amsmath}
\usepackage{amsthm}
\usepackage{amssymb}
\usepackage{amsfonts}
\usepackage{epsf}
\usepackage{graphicx}
\usepackage{hyperref}
\usepackage{indentfirst}
\usepackage{cases}

\newcommand{\supp}{\mathrm{supp}}

\newtheorem{theorem}{Theorem}[section]
\newtheorem{lem}[theorem]{Lemma}
\newtheorem{prop}[theorem]{Proposition}

\theoremstyle{definition}

\newcommand{\eq}{\begin{equation}}
\newcommand{\nq}{\end{equation}}
\newcommand{\eqa}{\begin{eqnarray}}
\newcommand{\nqa}{\end{eqnarray}}
\numberwithin{equation}{section}

\newtheorem{remark}{\indent \textrm{Remark}}

\begin{document}
\title{Asymptotic zero distribution of multiple
orthogonal polynomials associated with Macdonald functions}
\date{}
\author{Lun Zhang\footnotemark[1] , Pablo Rom\'an\footnotemark[2]  }
\maketitle {
\renewcommand{\thefootnote}{\fnsymbol{footnote}}
\footnotetext[1]{Department of Mathematics, Katholieke Universiteit
Leuven, Celestijnenlaan 200B, B-3001 Leuven, Belgium. E-mail:
lun.zhang@wis.kuleuven.be.}

\footnotetext[2]{Department of Mathematics, Katholieke Universiteit
Leuven, Celestijnenlaan 200B, B-3001 Leuven, Belgium. E-mail:
pablo.roman@wis.kuleuven.be.}

 }

\begin{abstract}
We study the asymptotic zero distribution of type II multiple
orthogonal polynomials associated with two Macdonald functions
(modified Bessel functions of the second kind). Based on the
four-term recurrence relation, it is shown that, after proper
scaling, the sequence of normalized zero counting measures converges
weakly to the first component of a vector of two measures which
satisfies a vector equilibrium problem with two external fields. We
also give the explicit formula for the equilibrium vector in terms
of solutions of an algebraic equation.
\end{abstract}

%----------------------------------------------------------------------
\section{Introduction }

Given a positive measure $\mu$ on the real line for which the support
is not finite and all moments exist, there exists a sequence of
monic orthogonal polynomials $\{p_k\}$ of degree $k$. Such polynomials
satisfy a three-term recurrence relation of the form
\begin{equation}\label{eq:recurrence_scalar_op}
xp_k(x)=p_{k+1}(x)+b_kp_k(x)+a_k^2p_{k-1}(x),\quad k\geq 0,
\end{equation}
with $a_k\geq 0$, $b_k\in \mathbb{R}$ and $p_0\equiv1$, $p_{-1}\equiv 0$.

It is well-known that the zeros of $p_k$ are real and simple. We can
associate with $p_k(x)$ the normalized zero counting measure
\begin{equation}\label{def: counting measure}
\nu(p_k)=\frac1k \sum_{j=1}^k \delta_{x_{j,k}},
\end{equation} where
$x_{j,k}$, $j=1,\ldots,k$, are the zeros of $p_k$ and $\delta_{x}$
denotes the Dirac point mass at $x$. A measure $\nu$ is called the
asymptotic zero distribution of $\{p_k\}$ if
$$\lim_{k\to \infty} \int f \, d\nu(p_k)=\int f \, d\nu$$
for every bounded continuous function $f$ on $\mathbb{R}$, i.e., it is the
weak limit of the measures $\nu(p_k)$.

Suppose $\lim_{k\to\infty}a_k=a$ and $\lim_{k\to\infty}b_k=b$ with $a>0$ and $b\in \mathbb{R}$, then the polynomials
generated by \eqref{eq:recurrence_scalar_op} have the asymptotic zero distribution $w_{[\alpha,\beta]}$ with density
\begin{equation}\label{w al be}
 \frac{dw_{[\alpha,\beta]}(x)}{dx}=\begin{cases} \frac{1}{\pi\sqrt{(\beta-x)(x-\alpha)}}, &x\in [\alpha,\beta],
 \\ 0, & \text{elsewhere}, \end{cases}
\end{equation}
where $\alpha=b-2a$, $\beta=b+2a$; cf. \cite{Nev}.

This result has been extended in \cite{KVA} to the case of orthogonal polynomials generated by the recurrence relation
$$xp_{k,n}(x)=p_{k+1,n}(x)+b_{k,n}p_{k,n}(x)+a_{k,n}^2p_{k-1,n}(x),\quad k,n\in\mathbb{N},$$
with varying recurrence coefficients $a_{k,n}>0$ and $b_{k,n}\in \mathbb{R}$
depending on a parameter $n$. Assume that the recurrence coefficients
have continuous limits
$$\lim_{k/n\to s}a_{k,n}=a(s), \quad \lim_{k/n\to s}b_{k,n}=b(s),$$
where $a:(0,\infty)\to [0,\infty)$, $b:(0,\infty)\to \mathbb{R}$ and the notation $\lim_{k/n \to s}$ means that both $k,n\to \infty$ with $k/n\to s>0$,
it is proved that the
asymptotic zero distribution is given by the average
\begin{equation}
\label{eq:zero_dist_OP}
\lim_{k/n\to s} \nu(p_{k,n})=\frac1s \int_0^s w_{[\alpha(s),\beta(s)]}ds,
\end{equation}
where $\alpha(s):=b(s)-2a(s)$, $\beta(s):=b(s)+2a(s)$ and $w_{[\alpha,\beta]}$
is defined by \eqref{w al be} if $\alpha<\beta$ and by $\delta_\alpha$ if $\alpha=\beta$;  see Theorem 1.10 of \cite{KVA}.

A natural generalization of this case consists in considering, for
each $n\in\mathbb{N}$, polynomials $P_{k,n}$ satisfying an $m$-term
recurrence relation with varying coefficients:
\begin{equation} \label{recurrence_m}
    xP_{k,n}(x) = P_{k+1,n}(x)+b_{k,n}^{(0)} P_{k,n}(x) +
    b_{k,n}^{(1)}P_{k-1,n}(x)+
        \cdots+b_{k,n}^{(m-2)}P_{k+2-m,n}(x),
\end{equation}
where $P_0\equiv 1$, $P_{-1}\equiv 0,\ldots,P_{-m+2}\equiv 0$ and
the recurrence coefficients have scaling limits
\[ \lim_{k/n\to s} b_{k,n}^{(j)} = b^{(j)}(s), \qquad j=0, \ldots, m-2, \]
for certain functions $b^{(0)}, \ldots, b^{(m-2)}$.

For the simplest case $b_{k,n}^{(j)} = b^{(j)}(s)$, i.e., we remove the dependence on
the parameter $n$ and the recurrence coefficients in \eqref{recurrence_m} are actually constant, the
zeros of $P_k=P_{k,n}$ are closely related to the spectrum of certain banded
Toeplitz matrix. Indeed, if we associate with the functions
$b^{(j)}$ a family of functions
\begin{equation} \label{sympbolsAs}
A_s(z) = z + b^{(0)}(s) + b^{(1)}(s)z^{-1}+ \cdots +b^{(m-2)}(s) z^{-m+2},
\end{equation}
and the sequence of $k \times k$ Toeplitz matrices $(T_k(A_s))_k$
with symbol $A_s$, defined by
\begin{equation}
\label{definition_toeplitz_matrix}
    (T_k(A_s))_{jl}=
    \begin{cases} 1,          & \text{if }l=j+1,\\
    b^{(i)}(s), & \text{if }l=j-i, \qquad i=0,
    \ldots, m-2,\\
    0,          & \text{otherwise},
    \end{cases}
\end{equation}
it is readily seen that $P_k(\lambda)=0$ if and only if $\lambda$ is
an eigenvalue of $T_k(A_s)$. Hence, the investigation of limiting zero
distribution of $P_k$ is equivalent to the study of the limiting
behavior of the spectrum of $T_k(A_s)$ as $k \to \infty$.

The limiting behavior of the spectrum of $T_k(A_s)$ as $k\to \infty$
is characterized by the solutions of the algebraic equation
$A_s(z)=x$; see \cite{BG}. For every $x\in\mathbb{C}$, there exist
exactly $m-1$ solutions of the equation $A_s(z)=x$ (assume that
$b^{(m-2)}(s) \neq 0$), which we denote by $z_j(x,s)$,
$j=1,\ldots,m-1$ and label these solutions by their absolute value
so that
\begin{equation}
\label{absolute_value_solutions}
    |z_{1}(x,s)| \geq |z_2(x,s)| \geq \cdots \geq |z_{m-1}(x,s)| > 0.
\end{equation}
We put
\begin{equation} \label{Gamma1s}
    \Gamma_1(s) = \{ x\in \mathbb{C} \mid |z_1(x,s)|=|z_2(x,s)| \},
    \end{equation}
which is a finite union of analytic arcs.

It was shown by Schmidt and Spitzer \cite{SS} that the eigenvalues of
$T_k(A_s)$ accumulate on the contour $\Gamma_1(s)$ as $k$ tends to
$\infty$. Moreover, Hirschman \cite{Hirschman} proved that the
sequence of normalized counting measures of the eigenvalues of
$T_k(A_s)$ converges weakly to a Borel probability measure $\mu^s_1$
supported on $\Gamma_1(s)$ as $k\to \infty$; see also
\cite[Chapter~11]{BG}. The precise form of $\mu^s_1$ is given by
\begin{equation} \label{measure_mu_for_recurrence}
    d\mu^s_1(x) = \frac{1}{2\pi i}
    \left( \frac{z'_{1-}(x,s)}{z_{1-}(x,s)}-\frac{z'_{1+}(x,s)}{z_{1+}(x,s)}
     \right)dx,
\end{equation}
which is due to the result in \cite{DK1}. Here, we have that $'$
denotes the derivative with respect to $x$, $dx$ is the complex line
element on $\Gamma_1(s)$ and $z_{1\pm}(x,s)$ is the limiting value
of $z_1(\tilde{x},s)$ as $\tilde{x}\to x$ from the $\pm$ side of
$\Gamma_1(s)$. Moreover, the measure $\mu_1^s$ is also characterized
by an equilibrium problem; see Theorem \ref{equilibrium_problem_Toeplitz}
below for a statement in the context of the specific example considered
in this paper.

Under certain conditions, the
polynomials $P_{k,n}$ satisfying the recurrence
\eqref{recurrence_m} have a limiting zero
distribution as well, which is an average, with respect to
the parameter $s$, of the measures
\eqref{measure_mu_for_recurrence}.
More precisely, we have (see Theorem 1.2 in
\cite{KR}):

\begin{theorem} \label{Theorem_zero_distribution_general}
Let for each $n\in \mathbb{N}$, $m-1$ sequences
$\{b_{k,n}^{(j)}\}_{k=0}^\infty$, $j=0,\ldots,m-2$, of real
coefficients be given and assume that there exist continuous
functions $b^{(j)}:[0,\infty)\to \mathbb{R}$, $j=0,\ldots,m-2$, such
that for each $s\geq 0$,
\begin{equation}
\label{Convergence_coefficients_m} \lim_{k/n\to s}
b^{(j)}_{k,n}=b^{(j)}(s),\quad j=0,\ldots,m-2.
\end{equation}
Let $P_{k,n}$ be the monic polynomials generated by the recurrence
\eqref{recurrence_m} and suppose that
\begin{itemize}
\item[\rm (a)]  The polynomials $P_{k,n}$ have real and simple zeros
$x_1^{k,n} < \cdots < x_k^{k,n}$
satisfying for each $k$ and $n$ the interlacing property
\[ x_j^{k+1,n}< x_j^{k,n}<x_{j+1}^{k+1,n}, \qquad \text{for } j=1,\ldots,k, \]
\item[\rm (b)] $\Gamma_1(s) \subset \mathbb R$ for every $s > 0$,
where $\Gamma_1(s)$
is given by \eqref{Gamma1s}.
\end{itemize}
Then the normalized zero counting measures $\nu(P_{k,n}) =
\frac{1}{k} \sum_{j=1}^k \delta_{x_j^{k,n}}$ have a weak limit as
$k, n \to \infty$ with $k/n \to \xi > 0$ given by
\begin{equation} \label{limitingmeasure}
    \lim_{k/n\to \xi} \nu(P_{k,n})= \frac{1}{\xi}\int_0^\xi \mu^{s}_1 \, ds,
    \end{equation}
where $\mu_1^s$ is the measure \eqref{measure_mu_for_recurrence}.
\end{theorem}

It is worth noting that in \cite{CCVA},
the authors present a conditional theorem giving the
asymptotic zero distribution for polynomials satisfying a specific
four-term recurrence relation.

The aim of this paper is to give more insight on the nature of
the asymptotic zero distribution in the particular case of
type II multiple orthogonal polynomials associated with two
Macdonald functions (modified Bessel functions of the second kind)
$K_\nu(x)$ ($\nu\geq 0$). A feature of the present case is the appearance
of a vector equilibrium problem with two
external fields, for which the first component of the unique minimizer
is the weak limit of the normalized counting zero
measures.

We mainly follow the idea in \cite{KR}, where the authors consider a model of
non-intersecting squared Bessel paths and derive a vector equilibrium problem
for the limiting zero distribution of type II multiple orthogonal polynomials
associated with
the modified Bessel functions of the first kind \cite{CVA1,CVA2}. In that case,
the vector equilibrium problem involves two measures supported on the positive
real line and the negative real line, respectively, with an external field
acting on the first measure and a constraint acting on the second measure
\cite[Theorem 1.7]{KR}.

%----------------------------------------------------------------------

\section{Statement of results}

\subsection{Multiple orthogonal polynomials associated with Macdonald functions}

Assuming $x>0$, we define the scaled Macdonald function $\rho_\nu$ by
\begin{equation}
\rho_\nu(x)=2x^{\nu/2}K_\nu(2\sqrt{x}),
\end{equation}
and consider two weights
\begin{equation}\label{two measures}
d\mu_1(x)=x^{\alpha}\rho_\nu(x) dx,\quad
d\mu_2(x)=x^{\alpha}\rho_{\nu+1}(x) dx,\quad \alpha>-1, \quad
\nu\geq 0,
\end{equation}
on the positive real line. For any $k,m\in\mathbb{N}$, the type II
multiple orthogonal polynomials $p_{k,m}^{\alpha}$ for the system of
weights $(\mu_1,\mu_2)$ are such that $p_{k,m}^{\alpha}$ is a monic
polynomial of degree $k+m$ and satisfies the following multiple
orthogonality conditions:
\begin{align}
\int_0^{\infty}p_{k,m}^{\alpha}(x)x^j d\mu_1(x)&=0, \quad
j=0,1,...,k-1,
\\
\int_0^{\infty}p_{k,m}^{\alpha}(x)x^j d\mu_2(x)&=0, \quad
j=0,1,...,m-1.
\end{align}
By taking $m=k$, we set
\begin{align*}
P_{2k}(x)=p_{k,k}^{\alpha}(x),\qquad
P_{2k+1}(x)=p_{k+1,k}^{\alpha}(x).
\end{align*}
An explicit formula for $P_k$ is given by
\begin{equation}\label{Pn}
P_k(x)=\sum_{j=0}^{k}a_k(j)x^{k-j},
\end{equation}
where
\begin{equation*}
a_k(j)=(-1)^j\binom{k}{j}\frac{(\alpha+1)_k(\alpha+\nu+1)_k}
{(\alpha+1)_{k-j}(\alpha+\nu+1)_{k-j}},\quad 0\leq j \leq k;
\end{equation*}
see \cite[Theorem 2]{CVA2001}. It is shown in \cite{VAY} that $P_k$
satisfies the following four-term recurrence relation
\begin{equation}
\label{eq:recurrence_relation}
x P_k(x) = P_{k+1}(x) + b_k P_k(x) + c_k P_{k-1}(x) + d_k P_{k-2}(x)
\end{equation}
with recurrence coefficients
\begin{equation} \label{eq:coefficientsbcdk}
\begin{aligned}
    b_{k} & = (k+\alpha+1)(3k+\alpha+2\nu) - (\alpha+1)(\nu-1), \\
    c_{k} & = k (k+\alpha) (k+\alpha+\nu) (3k+2\alpha+\nu), \\
    d_{k} & = k(k-1)(k+\alpha-1)(k+\alpha)(k+\alpha+\nu-1)(k+\alpha+\nu).
    \end{aligned}
\end{equation}

These polynomials constitute one of few examples of multiple orthogonal
polynomials that are not related to the classical orthogonal
polynomials. They are first introduced by Van Assche and Yakubovich
in \cite{VAY}, which solve an open problem posed by Prunikov
\cite{PRU}; see also \cite{CD,CCVA} for recent study.

Our goal is to investigate the limiting zero distribution of scaled
polynomials $P_{k}$. Namely, we introduce a new parameter $n\in
\mathbb{N}$ and put
\begin{equation}\label{scaled P(k,n)}
P_{k,n}(x):=\frac{P_k(n^2 x)}{n^{2k}}.
\end{equation}
Clearly, $P_{k,n}(x)$ is a polynomial of degree $k$ for each $n$. In
view of \eqref{eq:recurrence_relation}--\eqref{scaled P(k,n)}, it
is readily seen that $P_{k,n}(x)$ satisfies the following
recurrence relation
\begin{equation}
\label{eq:recurrence_P_doubly_indexed}
x P_{k,n}(x) = P_{k+1,n}(x) +
b_{k,n} P_{k,n}(x) + c_{k,n} P_{k-1,n}(x) + d_{k,n} P_{k-2,n}(x),
\end{equation}
with recurrence coefficients given by
\begin{equation} \label{eq:coefficientsbcdkn}
\begin{aligned}
    b_{k,n} & = \frac{(k+\alpha+1)(3k+\alpha+2\nu)-(\alpha+1)(\nu-1)}{n^2}, \\
    c_{k,n} & = \frac{k (k+\alpha) (k+\alpha+\nu)
    (3k+2\alpha+\nu)}{n^4}, \\
    d_{k,n} & = \frac{k(k-1)(k+\alpha-1)(k+\alpha)(k+\alpha+\nu-1)(k+\alpha+\nu)}
    {n^6}.
    \end{aligned}
    \end{equation}

As in \eqref{def: counting measure}, the normalized counting
zero measure of $P_{k,n}$ is defined by
\begin{equation}
\nu(P_{k,n})=\frac{1}{k}\sum_{P_{k,n}(x)=0} \delta_{x}.
\end{equation}
We will derive a vector equilibrium problem with two external fields
and show that the first component of the equilibrium vector is the
weak limit of $\nu(P_{k,n})$ as $k, n \to \infty$ with $k/n \to \xi
> 0$. The equilibrium vector itself can be explicitly given in terms
of the solutions of an algebraic equation. Our results are actually
rather general in the sense that we also allow the parameters
$\alpha$ and $\nu$ to increase proportionally to $n$ as $n$
increases.

%----------------------------------------------------------------
\subsection{Statement of results}\label{subsec:statement of results}
We scale the parameters $\alpha$ and $\nu$ in the following way:
\begin{equation}\label{scaling of parameters}
\alpha \mapsto pn, \qquad \nu \mapsto qn,
\end{equation}
with $p,q>0$. The results corresponding to $\alpha$ and $\nu$ fixed
can be obtained by taking the limits as $p,q\to 0$, respectively.

Let $k,n\to \infty$ in such a way that $k/n\to s$, for some $s\geq
0$, we then observe that the recurrence coefficients
\eqref{eq:coefficientsbcdkn} have scaling limits $b(s)$, $c(s)$ and
$d(s)$ given by
\begin{equation}
\label{scaling_limits2}
\begin{split}
\lim_{k/n \to s} b_{k,n}&=b(s)=3s^2+4sp+2sq+p^2+pq,\\
\lim_{k/n \to s} c_{k,n}&=c(s)=s(s+p)(s+p+q)(3s+2p+q),\\
\lim_{k/n \to s} d_{k,n}&=d(s)=s^2(s+p)^2(s+p+q)^2.
\end{split}
\end{equation}

Clearly, these limits depend on $p$ and $q$. As in
\eqref{sympbolsAs}, we have the associated family of symbols
\begin{align} \label{family_of_symbols0}
    A_s(z) & =z+b(s)+c(s)z^{-1}+d(s)z^{-2},
\end{align}
and the solutions $z_1(x,s)$, $z_2(x,s)$ and $z_3(x,s)$ of the
algebraic equation $A_s(z)= x$. We define $\Gamma_1(s)$ as in
\eqref{Gamma1s} and similarly
\begin{equation} \label{Gamma2s}
    \Gamma_2(s) = \{ x \in \mathbb{C} \mid |z_2(x,s)| = |z_3(x,s)| \}.
    \end{equation}

The following
proposition ensures that the polynomials $P_{k,n}$ satisfy the
hypothesis (a) of Theorem \ref{Theorem_zero_distribution_general}.

\begin{prop}
\label{Interlacing_Bk} Let $p,q >0$. Then the polynomials $P_{k,n}$
generated by \eqref{eq:recurrence_P_doubly_indexed} with recurrence
coefficients \eqref{eq:coefficientsbcdkn} have real and simple zeros
in $(0,\infty)$ with the interlacing property.
\end{prop}

\begin{proof}
The proof of the proposition follows from the fact that the measures $(d\mu_1,d\mu_2)$ from
\eqref{two measures} form an AT system (cf. \cite{NS,VAC,CCVA}),
which implies all the zeros of $P_k$ are simple, lie in
$(0,+\infty)$ \cite{NS,VAC}, and satisfy the interlacing property
\cite{AKLR}.
\end{proof}

By Proposition \ref{Proposition_Gamma_1_2_alpha} stated below, it is easily seen that the hypothesis (b) is also satisfied. Therefore, we see from
Theorem \ref{Theorem_zero_distribution_general} that, with the scaling given in \eqref{scaling of parameters}, the probability measure
\begin{align} \label{nu1xi-as-integral}
    \nu_1^{\xi} = \frac{1}{\xi} \int_0^{\xi} \mu_1^s \, ds, \qquad
    \xi>0,
\end{align}
is the weak limit of the normalized zero counting measures. The main result
of this paper is that $\nu_1^\xi$
can also be obtained as the first component of a vector of measures $(\nu_1^{\xi}, \nu_2^{\xi})$ that satisfies a vector equilibrium problem with two external
fields.

To define $\nu_2^{\xi}$, we need to introduce the second measure
$\mu_2^s$, which is supported on $\Gamma_2(s)$ (see \eqref{Gamma2s})
and given by
\begin{equation} \label{eq:densitymu2}
    d\mu^s_2(x) = \frac{1}{2\pi i}
    \left(\frac{z'_{2-}(x,s)}{z_{2-}(x,s)}-\frac{z'_{2+}(x,s)}{z_{2+}(x,s)}
    \right)dx, \qquad x \in \Gamma_2(s).
    \end{equation}
It is a positive measure on $\Gamma_2(s)$
with total mass $1/2$. For each $\xi>0$, we define $\nu_2^{\xi}$ in a manner similar to the
definition of $\nu_1^{\xi}$ in \eqref{nu1xi-as-integral}, i.e.,
\begin{align} \label{nu2xi-as-integral}
    \nu_2^{\xi} = \frac{1}{\xi} \int_0^{\xi} \mu_2^s \, ds.
    \end{align}
Then $\nu_2^{\xi}$ is a measure on
$\bigcup_{s<\xi}\Gamma_2(s)=\Gamma_2(\xi)$ with total mass $1/2$.

An essential point for the rest of the paper is
the main result of \cite{DK1}, which asserts that the vector of measures
$(\mu_1^s, \mu_2^s)$ is characterized by a vector equilibrium problem.
In the present context, this is stated in the following theorem.

\begin{theorem} \label{equilibrium_problem_Toeplitz}
For each $s > 0$, the vector $(\mu_1^s, \mu_2^s)$ is the unique
minimizer for the energy functional
\begin{multline}
\label{equilibrium_problem_mu1_mu2}
    \iint \log \frac{1}{|x-y|} d\mu_1(x)d\mu_1(y) +  \iint \log
    \frac{1}{|x-y|} d\mu_2(x)d\mu_2(y)\\
    - \iint \log \frac{1}{|x-y|}d\mu_1(x)d\mu_2(y)
\end{multline}
among all vectors $(\mu_1, \mu_2)$ satisfying $\supp(\mu_j) \subset
\Gamma_j(s)$ for $j=1,2$, and
\[ \int d \mu_1 = 1, \qquad \int d\mu_2 = \frac{1}{2}. \]

The measures $\mu_1^s$ and $\mu_2^s$ satisfy the following
Euler-Lagrange variational conditions:
\begin{align}
\label{variational_mu_1} 2\int \log |x-y|d\mu^s_1(y)-\int \log
|x-y|d\mu^s_{2}(y) & =\ell^s, \qquad x\in \Gamma_1(s),
\end{align}
for some constant $\ell^s$, and
\begin{align}\label{variational_mu_2} 2\int \log
|x-y|d\mu^s_2(y)-\int \log |x-y|d\mu^s_{1}(y) & =0, \qquad \, x\in
\Gamma_2(s).
\end{align}
\end{theorem}

We shall obtain the equilibrium problem for the vector of measures
$(\nu_1^\xi,\nu_2^\xi)$ by integrating the variational conditions
\eqref{variational_mu_1} and \eqref{variational_mu_2} with respect
to the variable $s$. The main difficulty lies in the fact that
$\Gamma_1(s)$ and $\Gamma_2(s)$ are varying with $s$.
Hence, it is necessary to study how these contours depend on $s$.
The next proposition reveals that $\Gamma_1(s)$ and $\Gamma_2(s)$ are
indeed real intervals and actually are increasing as $s$ increases.
The monotonicity of $\Gamma_1(s)$ and $\Gamma_2(s)$, as we will see
later, plays a important role in
the derivation of the equilibrium problem.

\begin{prop}
\label{Proposition_Gamma_1_2_alpha} For each $s
> 0$, we have that $\Gamma_1(s) \subset (0,\infty)$ and $\Gamma_2(s)
\subset (-\infty,0)$. More precisely, there exist $\eta(s) < 0 <
\beta(s) < \gamma(s) $ so that
\begin{align} \label{Gamma12s-in-second-scaling}
    \Gamma_1(s) = [\beta(s), \gamma(s)], \qquad
    \Gamma_2(s) = (-\infty, \eta(s)].
    \end{align}
In addition, we have
\begin{enumerate}
\item[\rm (a)] $\gamma(s)$ is positive and strictly increasing for $s > 0$,
with $\lim_{s\to 0+} \gamma(s)=p(p+q)$  and $\lim_{s\to \infty}
\gamma(s)=\infty$,
\item[\rm (b)] $\beta(s)$ is positive and strictly
decreasing for $s > 0$ with
$\lim_{s\to 0+} \beta(s)=p(p+q)$ and $\lim_{s\to \infty}
\beta(s)=0$,
\item[\rm (c)] $\eta(s)$ is
negative and strictly increasing for $s > 0$ with $\lim_{s\to 0+}
\eta(s)=-q^2/4$ and $\lim_{s\to \infty} \eta(s)=0$.
\end{enumerate}
\end{prop}

Figure \ref{figure_curves_gamma} gives an illustrative plot of the functions
$\beta(s)$, $\gamma(s)$ and $\eta(s)$.

\begin{figure}[t]
\centering
\begin{overpic}[width=5cm,height=5cm]{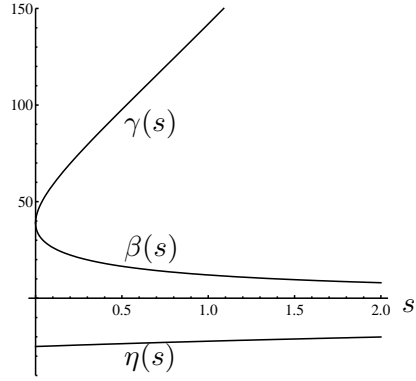}
\put(103,17){$s$} \put(30,3){$\eta(s)$} \put(30,32){$\beta(s)$}
\put(30,65){$\gamma(s)$}
\end{overpic}
\caption{\label{figure_curves_gamma}Graph of the curves $\beta(s)$,
$\gamma(s)$ and $\eta(s)$ with $p=3$ and $q=10$.}
\end{figure}

Now we come to the main result of this paper, i.e.,
the equilibrium problem for the vector of measures
$(\nu_1^\xi,\nu_2^\xi)$. The fact that $\Gamma_1(s)$ and $\Gamma_2(s)$ are increasing as $s$ increases,
induces two external fields $V_1(x)$ and $V_2(x)$ acting on $\nu_1^\xi$ and $\nu_2^\xi$, respectively.

\begin{theorem} \label{proposition_V(x)_nu_1}
For every $\xi > 0$, the vector of measures $(\nu^\xi_1,\nu^\xi_2)$
is the unique minimizer for the energy functional
\begin{align} \label{functionalwithV}
    \iint &\log \frac{1}{|x-y|} d\nu_1(x)d\nu_1(y) +
    \iint \log \frac{1}{|x-y|}d\nu_2(x)d\nu_2(y) \nonumber \\
    &-\iint \log \frac{1}{|x-y|} d\nu_1(x)d\nu_2(y)
    +\frac{1}{\xi} \int V_1(x) d\nu_1(x)+\frac{1}{\xi} \int V_2(x) d\nu_2(x),
\end{align}
over all vectors of measures $(\nu_1,\nu_2)$ such that
$\supp(\nu_1)\subset [0,\infty)$, $\int d\nu_1 = 1$ and
$\supp(\nu_2)\subset (-\infty,0]$, $\int d\nu_2 = 1/2$, where
\begin{align} \label{definition-of-V}
    V_1(x) = \int_0^{\infty} \log \left|\frac{z_1(x,s)}{z_2(x,s)} \right| \,
    ds \quad \textrm{and} \quad V_2(x) = \int_0^{\infty} \log \left|\frac{z_2(x,s)}{z_3(x,s)}
    \right| \, ds.
    \end{align}

The measures $\nu_1^{\xi}$ and $\nu_2^{\xi}$ are characterized by
the following variational conditions:
\begin{align}
\label{variational_nu_1} 2\int \log |x-y| d\nu^\xi_1(y) -\int \log
|x-y| d\nu^\xi_2(y) - \frac1\xi V_1(x)
    &\begin{cases} = \ell, \quad \text{for } x\in  \supp (\nu^\xi_1),  \\
    \leq \ell, \quad \text{for } x\in [0,\infty), \end{cases}
\end{align}
for some $\ell$, and
    \begin{align}
\label{variational_nu_2} 2\int \log |x-y| d\nu^\xi_2(y) -  \int \log
|x-y| d\nu^\xi_1(y) - \frac1\xi V_2(x) &
    \begin{cases} =0, \quad \text{for } x \in \operatorname{supp}(\nu^\xi_2), \\
    \leq 0,\quad \text{for } x \in (-\infty,0].
    \end{cases}
\end{align}
\end{theorem}

Since the usual equilibrium problems provide a powerful tool in the
asymptotic study of orthogonal polynomials (cf.
\cite{DKMVZ992,DKMVZ991,ST}), we hope the vector equilibrium problem
stated above will be helpful in further investigation of the
asymptotics of $P_{k}$ in \eqref{Pn}; see also \cite{DK2,KMW} for a
recent applications of vector equilibrium problems in some random
models.

Finally, we give the explicit formulas of the external fields $V_1$
and $V_2$ defined in \eqref{definition-of-V}, and the densities of
the measures $\nu_1^\xi$, $\nu_2^\xi$ given in
\eqref{nu1xi-as-integral} and \eqref{nu2xi-as-integral},
respectively.
\begin{theorem} \label{theorem:external_field}
For every $p,q > 0$, we have
\begin{multline}
\label{external_field}
    V_1(x) = \sqrt{q^2+4x}-p\log(4x)-q\log(\sqrt{q^2+4x}+q)
    \\ -2p-q+p\log(4p^2+4pq)+q\log(2p+2q),
\end{multline}
and
\begin{equation}
\label{external_field V2}
V_2(x) = \left\{
\begin{array}{ll}
0, & \hbox{for $x<-q^2/4$,} \\
-2\sqrt{q^2+4x}+
q\log\Big(\frac{q+\sqrt{q^2+4x}}{q-\sqrt{q^2+4x}}\Big), & \hbox{for
$-q^2/4 \leq x <0$.}
\end{array}
\right.
\end{equation}
%\end{theorem}
%\begin{theorem}
%\label{theorem:densities_measures}
The densities of the measures
$\nu_1^\xi$, $\nu_2^\xi$ are given by
\begin{align}
\label{eq:density_nu1_formula}
\frac{d\nu_1^\xi}{dx}(x)&=c\frac{(z_{1+}(x,s)-z_{1-}(x,s))}
{z_{1+}(x,s)z_{1-}(x,s)}, \quad for \, x\in\supp(\nu_1^\xi),\\
\label{eq:density_nu2_formula}
\frac{d\nu_2^\xi}{dx}(x)&=\begin{cases}
c\frac{(z_{2+}(x,s)-z_{2-}(x,s))}{z_{2+}(x,s)z_{2-}(x,s)x},
\quad & for \, x\in \supp(\nu_2^\xi) ~and~ -q^2/4<x<0, \\
c \frac{(z_{2+}(x,s)-z_{2-}(x,s))}{z_{2+}(x,s)z_{2-}(x,s)x} +
\frac{\sqrt{|q^2+4x|}}{2\pi x }, \quad & for \, x\in
\supp(\nu_2^\xi) ~and~ x\leq-q^2/4,
\end{cases}
\end{align}
where
$$c=\frac{\xi(\xi+p)(\xi+p+q)}{2 \pi i}.$$
\end{theorem}
Figure \ref{figure_densities} shows the graph of the densities of
$\nu_1^\xi$ and $\nu_2^\xi$.
\begin{figure}[t]
\centering
\begin{overpic}[width=12cm,height=6cm]{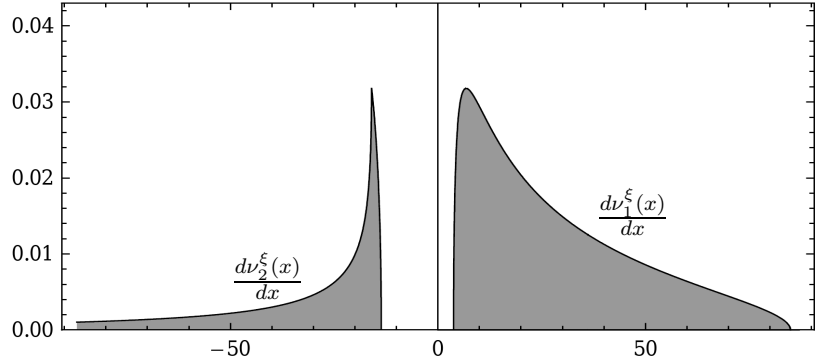}
\put(30,13){$\frac{d\nu_2^\xi(x)}{dx}$}
\put(70,20){$\frac{d\nu_1^\xi(x)}{dx}$}
\end{overpic}
\caption{\label{figure_densities} The densities of the measures
$\nu_1^\xi$ and $\nu_2^\xi$. ($\xi$=1, $p$=1.7, $q$=8)}
\end{figure}
%Figure \ref{figure_external_field} shows the graph of the external
%fields $V_1(x)$ and $V_2(x)$.

\begin{remark}
If we take $p,q\to 0$, which corresponds to fixed parameters
$\alpha$ and $\nu$, the symbol \eqref{family_of_symbols0} becomes
$$A_s(z)=z+3s^2+\frac{3s^4}{z}+\frac{s^6}{z^2}.$$
An straightforward calculation using \eqref{eq:density_nu1_formula}
gives
$$\frac{d\nu_1^\xi}{dx}(x)=\begin{cases}
\tfrac{4}{27\xi^2}h(\tfrac{4x}{27\xi^2}), & x\in(0,\tfrac{27\xi^2}{4}), \\
0, &elsewhere,\end{cases}$$ with
$$h(y)=\frac{3\sqrt{3}}{4\pi}
\frac{(1+\sqrt{1-y})^{1/3}-(1-\sqrt{1-y})^{1/3}}{y^{2/3}},$$ which
agrees with Theorem 2.7 in \cite{CCVA}. This case is illustrated in
Figure \ref{figure_pq0}.
\end{remark}

\begin{figure}[t]
\centering
\begin{overpic}[width=12cm,height=6cm]{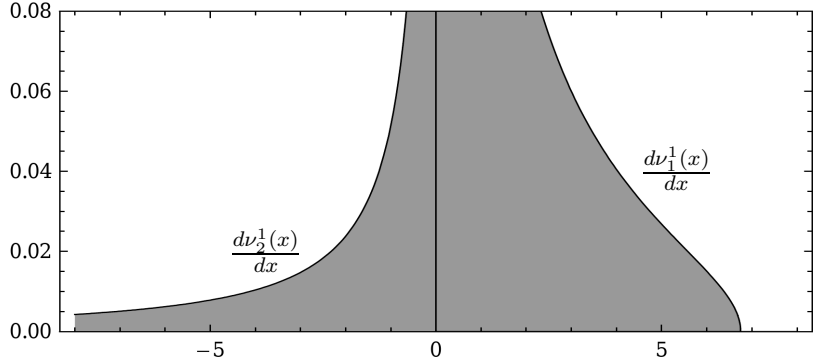}
\put(30,16){$\frac{d\nu_2^1(x)}{dx}$}
\put(75,25){$\frac{d\nu_1^1(x)}{dx}$}
\end{overpic}
\caption{\label{figure_pq0}The densities of the measures
$\nu_1^1$ and $\nu_2^1$ in the case $p=q=0$.}
\end{figure}

The rest of this paper is organized as follows. We first prove
Proposition \ref{Proposition_Gamma_1_2_alpha} in Section
\ref{sec:proof_propositions}. The proof of
Theorem \ref{proposition_V(x)_nu_1} is given in Section
\ref{sec:proof_theorem_Var_conditions}. We conclude this paper with
the proof of Theorem \ref{theorem:external_field}, where we use a
nonlinear transformation to evaluate the integrals used to define
the external fields and the equilibrium vector.

%--------------------------------------------------------------
\section{Proof of Proposition \ref{Proposition_Gamma_1_2_alpha}}
\label{sec:proof_propositions}

The symbol \eqref{family_of_symbols0} with the functions $b(s)$,
$c(s)$ and $d(s)$ from \eqref{scaling_limits2} allows for a
factorization
\begin{align} \label{eq:factorization_As}
    A_s(z) = \frac{(z+s(s+p))(z+s(s+p+q))(z+(s+p)(s+p+q))}{z^2}.
\end{align}
By \eqref{eq:factorization_As}, it follows that $A_s$ has three negative
simple zeros $ r_1, r_2, r_3$. We order them so that
\[ r_1 < r_2 < r_3 < 0; \]
see Figure \ref{Plot_symbol} for the graph of $A_s(z)$.

The derivative of $A_s(z)$,
$$A'_s(z) = 1 - c(s) z^{-2} - 2 d(s) z^{-3}, $$
has three roots in the complex plane. From Figure~\ref{Plot_symbol},
we see that all zeros of $A_s'$ are real. We denote the zeros of
$A'(z)$ by $y_1$, $y_2$ and $y_3$ so that
\begin{equation}\label{y_i}
 y_1 < y_2 < 0< y_3,
\end{equation}
as indicated in Figure \ref{Plot_symbol}. To emphasize the
dependence on $s$, we also write $y_1(s)$, $y_2(s)$ and $y_3(s)$.

\begin{figure}[h]
\begin{center}
\includegraphics[scale=1]{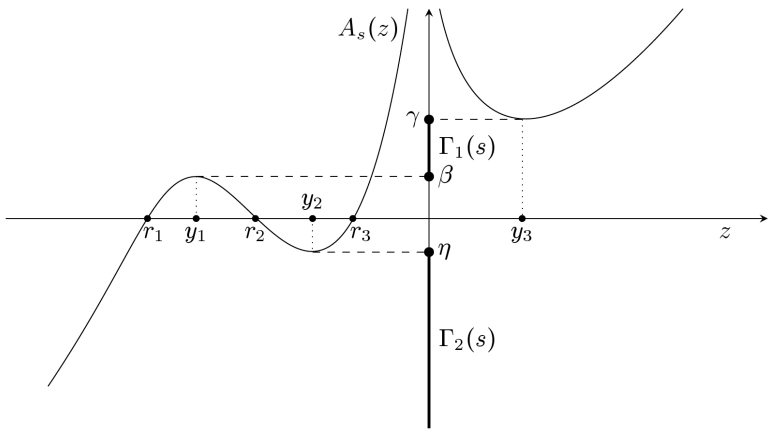}
\caption{\label{Plot_symbol} The graph of $A_s(z)$
\eqref{sympbolsAs}, $z\in \mathbb{R}$.}
\end{center}
\end{figure}

Before proving Proposition \ref{Proposition_Gamma_1_2_alpha} we
first need two lemmas. The fact that all the zeros $r_j$ of the
symbol $A_s$ are strictly negative plays a key role in these
proofs.

\begin{lem} \label{lem:x_in_R}
Assume that $z_1, z_2\in \mathbb{C}$ are such that $z_1\neq z_2$,
$|z_1|=|z_2|$ and $A_s(z_1)=A_s(z_2)=x$. Then $z_1=\bar z_2$ and
$x\in \mathbb{R}$.
\end{lem}
\begin{proof}
This lemma is essentially Lemma 4.1 in \cite{KR}. For the
convenience of the readers, we repeat the proof here.

Since $z_1\neq z_2$ and $|z_1|=|z_2|$, we may assume $z_1=\rho
e^{i\theta_1} $, $z_2=\rho e^{i\theta_2} $ with $\rho>0$ and
$\theta_1\neq\theta_2$, $\theta_{1,2}\in[-\pi,\pi]$. Since the zeros
$r_j$ of $A_s$ are strictly negative, it is easily seen that
$f(\theta):=|A_s(\rho e^{i \theta})|$ is an even function on
$[-\pi,\pi]$, which is strictly decreasing as $\theta$ increases
from $0$ to $\pi$. Hence, the equality
\[  | A_s(\rho e^{i \theta_1})| = | A_s(\rho e^{i \theta_2})| \]
holds if and only if $\theta_2 = - \theta_1$. It then follows that
$z_1 = \bar{z}_2$, and
 \[ x=A_s(z_1)=\overline{A_s(\bar z_1)}=\overline{A_s(z_2)}=\overline{x}, \]
 so that $x\in \mathbb{R}$.
\end{proof}

\begin{lem}
\label{lem:Gammas_real} For each $s>0$, we have
$\Gamma_1(s)\cup\Gamma_2(s)\subset \mathbb{R}$ and
$\Gamma_1(s)\cap\Gamma_2(s)=\emptyset$.
\end{lem}
\begin{proof}
To show $\Gamma_1(s)\cup\Gamma_2(s)\subset \mathbb{R}$, we consider
two cases, based on whether the equation $A_s(z)-x=0$ has a double
root or not. If $x\in \Gamma_1(s)\cup\Gamma_2(s)$ and $A_s(z)-x=0$
has a double root, then there exists $z_1\in \mathbb{C}$ such that
$A_s(z_1)=x$ and $A'_s(z_1)=0$. Since all zeros of $A'_s(z)$ are
real, it follows that $x\in \mathbb{R}$. On the other hand, suppose
$x\in \Gamma_1(s)\cup\Gamma_2(s)$ and $A_s(z)-x=0$ does not have a
double root, then there exist $z_1,z_2\in \mathbb{C}$ such that
$z_1\neq z_2$, $|z_1|=|z_2|$ and $x=A_s(z_1)=A_s(z_2)$. We then
conclude form Lemma~\ref{lem:x_in_R} that $x\in \mathbb{R}$. This
proves that $\Gamma_1(s)\cup\Gamma_2(s)\subset \mathbb{R}$.

To show $\Gamma_1(s)\cap\Gamma_2(s)=\emptyset$, we observe that, if
$x\in \Gamma_1(s)\cap\Gamma_2(s)$, there exist three solutions of
$A_s(z)=x$, one negative solution $z_1<0$ and two complex conjugated
solutions $z_2$ and $\bar z_2$ . Moreover $z_1\neq z_2$ and $z_1\neq
\bar z_2$. Now $|z_1|=|z_2|$ ($x\in\Gamma_1(s)$) and
 $A_s(z_1)=A_s(z_2)$. Again, by Lemma
\ref{lem:x_in_R}, we obtain $z_1=\bar z_2$, which is a
contradiction. Therefore, $\Gamma_1(s)\cap\Gamma_2(s)=\emptyset$.
\end{proof}

\paragraph{Proof of Proposition \ref{Proposition_Gamma_1_2_alpha}.}
Since $p$ and $q$ are positive, there are three local extrema of
$A_s(z)$, namely $\beta(s), \gamma(s), \eta(s)$, such that
$$\eta(s)<0<\beta(s)<\gamma(s).$$
If $x\in(\eta(s),\beta(s))\cup(\gamma(s),\infty)$, there exist three
different real solutions of $A_s(z)=x$. These solutions differ in
absolute value, which is obvious if $x\in(\eta(s),\beta(s))$ (cf.
Figure \ref{Plot_symbol}) and a consequence of Lemma
\ref{lem:x_in_R} if $x\in (\gamma(s),\infty)$. On the other hand,
there is one real and two complex conjugated solutions whenever
$x\in(-\infty,\eta(s)]\cup[\beta(s),\gamma(s)]$. Therefore
$$\Gamma_1(s)\cup\Gamma_2(s)\subset (-\infty,\eta(s)]\cup[\beta(s),\gamma(s)].$$
Note that $A_s(z)=\gamma(s)$ has a double root at $y_3=y_3(s)>0$ and
one negative root whose absolute value is less than $y_3(s)$. Thus
$\gamma(s) \in \Gamma_1(s)$. By the same argument, we see
$\beta(s)\in \Gamma_1(s)$. In view of the fact that $\Gamma_1(s)$ is
connected (see \cite{Ullman},\cite[Theorem 11.19]{BG}), it then
follows that $\Gamma_1(s)=[\beta(s),\gamma(s)]$. Similarly, we
notice that $A_s(z)=\eta(s)$ has a double root at $y_2(s)<0$ and a
negative root whose absolute value is larger than $|y_2(s)|$.
Therefore $\eta_2(s)\in \Gamma_2(s)$ and
$\Gamma_2(s)=(-\infty,\eta(s)]$.

To show that $\gamma(s)$ is an increasing function, we introduce
\begin{equation}\label{B(z,s)}
B(z,s)=A_s\Big(\frac{s(s+p)(s+p+q)}{z-s}\Big)=
\frac{z(z+p)(z+p+q)}{z-s}.
\end{equation}
Taking the partial derivative of $B(z,s)$ with respect to $s$, we
obtain
\begin{equation}\label{partial s of B}
\frac{\partial B(z,s)}{\partial s}=\frac{z(z+p)(z+p+q)}{(z-s)^2}.
\end{equation}
As a function of $z$, it is easily seen that $B(z,s)$ has a local
minimum at $s+s(s+p)(s+p+q)/y_3(s)$ and
$$\gamma(s)=A_s(y_3(s))=B(s+s(s+p)(s+p+q)/y_3(s),s).$$
This, together with \eqref{partial s of B}, implies that
\begin{align}\label{derivative of gamma s}
\gamma'(s)&=\frac{\partial B(s+s(s+p)(s+p+q)/y_3(s),s)}{\partial
z}\frac{\partial (s+s(s+p)(s+p+q)/y_3(s))} {\partial s} \nonumber \\
&~~~~~+\frac{\partial B(s+s(s+p)(s+p+q)/y_3(s),s)}{\partial
s}\nonumber \\
&=\frac{\partial B(s+s(s+p)(s+p+q)/y_3(s),s)}{\partial s} \nonumber \\
&=\frac{\hat z_3(s)(\hat z_3(s)+p)(\hat z_3(s)+p+q)}{(\hat
z_3(s)-s)^2},
\end{align}
where $\hat z_3(s)=s+s(s+p)(s+p+q)/y_3(s)$. As $y_3(s)>0$ and
$p,q>0$, we have $\hat z_3(s)>0$ for all $s>0$. Therefore
$\gamma(s)$ is increasing on $(0,\infty)$ by \eqref{derivative of
gamma s}. The monotonicity of $\beta(s)$ and $\eta(s)$ can be proved
in similar manners. Indeed, by the same argument, we have
\begin{align}
\eta'(s)&=\frac{\hat z_2(s)(\hat z_2(s)+p)(\hat z_2(s)+p+q)}{(\hat
z_2(s)-s)^2}, \label{derivative of eta s}\\
\beta'(s)&=\frac{\hat z_1(s)(\hat z_1(s)+p)(\hat z_1(s)+p+q)}{(\hat
z_1(s)-s)^2},
\end{align}
where $\hat z_j(s)=s+s(s+p)(s+p+q)/y_j(s)$, $j=1,2$. Note that
$y_2(s)$ and $y_1(s)$ are local extreme points of $A_s(z)$, whose
zeros are $-s(s+p)$, $-s(s+p+q)$ and $-(s+p)(s+p+q)$. Therefore, in
view of \eqref{y_i} (see also Figure \ref{Plot_symbol}), we have
\begin{equation*}
-s(s+p+q)<y_2(s)<-s(s+p) ~~\textrm{and}~~
-(s+p)(s+p+q)<y_1(s)<-s(s+p+q),
\end{equation*}
which implies
\begin{equation}\label{range of z1 z2}
-(p+q)<\hat z_2(s)<-p ~~\textrm{and}~~ -p<\hat z_1(s)<0.
\end{equation}
Combining \eqref{derivative of eta s}--\eqref{range of z1 z2}, it
follows that $\eta'(s)>0$ and $\beta'(s)<0$ for all $s>0$, which
gives the desired monotonicity of $\beta(s)$ and $\eta(s)$.

Finally, we come to the boundary values of $\gamma(s)$, $\beta(s)$
and $\eta(s)$. A straightforward calculation yields that $y_1(s)$,
$y_2(s)$ and $y_3(s)$ have the following behavior as $s\rightarrow
0$
\begin{align*}
y_1(s)&=-\sqrt{p(p+q)(2p+q)s}+\mathcal{O}(s),\\
y_2(s)&=-\frac{2p(p+q)}{2p+q}s+\mathcal{O}(s^2),\\
y_3(s)&=\sqrt{p(p+q)(2p+q)s}+\mathcal{O}(s).\\
\end{align*}
Hence,
\begin{align*}
\gamma(s)=A_s(y_3(s))&=p(p+q)+\mathcal{O}(s),\\
\eta(s)=A_s(y_2(s))&=-\frac14 q^2+\mathcal{O}(s),\\
\beta(s)=A_s(y_1(s))&=p(p+q)+\mathcal{O}(s),
\end{align*}
as $s\to 0+$. This proves the limits
$$\lim_{s\rightarrow 0+} \beta(s) = \lim_{s\rightarrow 0+} \gamma(s)=p(p+q)
~~ \textrm{and}~~ \lim_{s\rightarrow 0+} \eta(s)=-\frac14 q^2.$$

On the other hand, note that $y_1(s)$, $y_2(s)$ and $y_3(s)$ have
the following behavior as $s\rightarrow \infty$,
\begin{align*}
y_1(s)&=-s^2+c_1s+\mathcal{O}(1),\\
y_2(s)&=-s^2+c_2s+\mathcal{O}(1),\\
y_3(s)&=2s^2+\frac{8p-4q}{3}s+\mathcal{O}(1),\\
\end{align*}
where $c_1=(-4p-2q-\sqrt{p^2+pq+q^2})/3$ and
$c_2=(-4p-2q+\sqrt{p^2+pq+q^2})/3$, the limits of $\gamma(s)$,
$\beta(s)$ and $\eta(s)$ as $s\to \infty$ can be obtained in a
manner similar to the situation $s\to 0+$, we omit the details and
this completes the proof of the Proposition
\ref{Proposition_Gamma_1_2_alpha}. \qed

\section{Proof of Theorem~\ref{proposition_V(x)_nu_1}}
\label{sec:proof_theorem_Var_conditions}

From the definitions of $\nu_1^{\xi}$ and $\nu_2^{\xi}$ in
\eqref{nu1xi-as-integral} and \eqref{nu2xi-as-integral}, it is clear
that $\supp(\nu_1^{\xi}) \subset [0,\infty)$, $\int d\nu_1^{\xi} =
1$ and $\supp(\nu_2^{\xi}) \subset (-\infty,0]$, $\int d\nu_2^{\xi}
= 1/2$. Indeed, on account of the fact that the sets $\Gamma_1(s) =
\supp(\mu_1^s)$ and $\Gamma_2(s) = \supp(\mu_2^s)$ are increasing as
$s$ increases (see Proposition \ref{Proposition_Gamma_1_2_alpha}),
it follows that
\begin{equation} \label{eq:supportnu1}
    \supp(\nu_1^\xi) = \bigcup_{s \leq \xi} \Gamma_1(s) = \Gamma_1(\xi)
\end{equation}
and
\begin{equation}\label{eq:supportnu2}
\supp(\nu_2^\xi) = \bigcup_{s \leq \xi} \Gamma_2(s) = \Gamma_2(\xi).
\end{equation}

Thus, in order to show that $(\nu_1^{\xi}, \nu_2^{\xi})$ is the
minimizer of the energy functional \eqref{functionalwithV} under the
conditions stated in Theorem~\ref{proposition_V(x)_nu_1}, it
suffices to prove that the vector $(\nu_1^{\xi}, \nu_2^{\xi})$
satisfies the variational conditions \eqref{variational_nu_1} and
\eqref{variational_nu_2}.

The basic idea is, as mentioned in Subsection \ref{subsec:statement of results},
to integrate the variational conditions
\eqref{variational_mu_1} and \eqref{variational_mu_2} with respect
to $s$ from $0$ to $\xi$. Here, we need a more general expression
for the variational conditions \eqref{variational_mu_1} and
\eqref{variational_mu_2}, namely
\begin{align}
\label{variational_mu_1_log}
2\int \log |x-y| d\mu^s_1(y) -  \int \log |x-y| d\mu^s_2(y) - \ell^s & = \log \left|
\frac{z_1(x,s)}{z_2(x,s)} \right|,\\
\label{variational_mu_2_log} 2\int \log |x-y| d\mu^s_2(y) -  \int
\log |x-y| d\mu^s_1(y) &= \log \left| \frac{z_2(x,s)}{z_3(x,s)}
\right|,
\end{align}
for all $x\in\mathbb{C}$, which are contained in the proof of Theorem 2.3 of \cite{DK1}.
These conditions reduce to \eqref{variational_mu_1} and \eqref{variational_mu_2} whenever
$x\in \Gamma_1(s)$ and $x\in \Gamma_2(s)$, respectively.

\paragraph{Proofs of \eqref{variational_nu_1} and \eqref{variational_nu_2}.}
Multiplying both sides of \eqref{variational_mu_1_log} by $1/\xi$
and integrating with respect to $s$ from $0$ to $\xi$, we obtain
from interchanging the order of integration that
\begin{equation} \label{eq:integrated1}
    2\int \log |x-y| d\nu^\xi_1(y) -  \int \log |x-y| d\nu^\xi_2(y) -
    \ell =\frac1\xi \int_0^\xi \log \left| \frac{z_1(x,s)}{z_2(x,s)}
    \right| ds,
    \end{equation}
for all $x\in\mathbb{C}$ and some constant $\ell\in \mathbb{R}$. The
measures $\nu_1^\xi$ and $\nu_2^\xi$ in \eqref{eq:integrated1} are
defined in \eqref{nu1xi-as-integral} and \eqref{nu2xi-as-integral},
respectively.

Let $x \geq 0$. Since $|z_1(x,s)| \geq |z_2(x,s)|$ for every $s$, it
follows from \eqref{eq:integrated1} and \eqref{definition-of-V} that
\begin{equation} \label{eq:integrated2}
    2\int \log |x-y| d\nu^\xi_1(y) -  \int \log |x-y| d\nu^\xi_2(y) -
    \ell \leq \frac{1}{\xi} \int_0^{\infty} \log \left| \frac{z_1(x,s)}{z_2(x,s)}
    \right| ds = \frac{1}{\xi} V_1(x).
    \end{equation}
Suppose now $x \in \supp(\nu_1^{\xi})$, then $x \in \Gamma_1(\xi)$
by \eqref{eq:supportnu1}, and therefore $x \in \Gamma_1(s)$ for
every $s \geq \xi$, since the sets are increasing. Thus $|z_1(x,s)|
= |z_2(x,s)|$ for every $s \geq \xi$, and equality holds in
\eqref{eq:integrated2} for $x \in \supp(\nu_1^{\xi})$. This
completes the proof of \eqref{variational_nu_1}.

The variational condition \eqref{variational_nu_2} for $\nu_2^{\xi}$
can be proved in a manner similar to \eqref{variational_nu_1} by
using \eqref{variational_mu_2_log} and \eqref{eq:supportnu2}, we
omit the details. \qed

\section{Proof of Theorem~\ref{theorem:external_field}}
\label{sec:externalfield}

We conclude this paper with the proof of
Theorem~\ref{theorem:external_field}. We start with the following
lemma that embodies the behavior of three solutions of $A_s(z)$ as
$s\to 0+$.
\begin{lem}
\label{asymptotic_solutions_alpha} Let $A_s(z)$ be given by
\eqref{eq:factorization_As}, and let $z_1(x,s)$, $z_2(x,s)$ and
$z_3(x,s)$ be the solutions of $A_s(z)=x$, ordered as in
\eqref{absolute_value_solutions}. Then
\begin{equation} \label{eq:zforsto0}
\begin{aligned}
\lim_{s\to 0+} z_1(x,s) & = x-p(p+q),\\
\lim_{s\to 0+} s^{-1}z_2(x,s) & =
-\frac{p(p+q)(2p+q+\sqrt{q^2+4x})}{2(p^2+pq-x)},\\
\lim_{s\to 0+} s^{-1}z_3(x,s) & =
-\frac{p(p+q)(2p+q-\sqrt{q^2+4x})}{2(p^2+pq-x)}.
\end{aligned}
\end{equation}
\end{lem}
\begin{proof}
The lemma follows by a straightforward computation.
\end{proof}

To prove Theorem~\ref{theorem:external_field}, we need to establish
the identities
\eqref{external_field}--\eqref{eq:density_nu2_formula}.

\paragraph{Proof of \eqref{external_field}.}
Let $x > 0$. From Proposition \ref{Proposition_Gamma_1_2_alpha}, it
follows that there exists a unique $s^*(x) \geq 0$ so that for all
$s > 0$,
\begin{equation}\label{def of s star 1}
x \in \Gamma_1(s) \quad \Longleftrightarrow \quad s \geq s^*(x).
\end{equation}
Then $\log |z_1(x,s)/z_2(x,s)|=0$ for all $s\geq s^*(x)$, and so by
\eqref{definition-of-V}
\begin{equation} \label{eq:finiteintegralV1}
V_1(x)=\int_0^{s^*(x)}  \log \left| \frac{z_1(x,s)}{z_2(x,s)} \right| ds.
\end{equation}
There is a special value
\[ x_0 = p(p+q) = \lim_{s \to 0+} \beta(s) = \lim_{s \to 0+} \gamma(s) \]
that belongs to every $\Gamma_1(s)$ for any $s > 0$. Then $s^*(x_0)
= 0$ and
\begin{equation} \label{eq:Vinx0}
V_1(x_0) = 0.
\end{equation}

The derivative of \eqref{eq:finiteintegralV1} is
\begin{equation} \label{integral_derivative_A}
V_1'(x) = \int_0^{s^*(x)} \left(\frac{1}{z_1(x,s)}\frac{\partial
z_1(x,s)}{\partial x} -\frac{1}{z_2(x,s)}\frac{\partial
z_2(x,s)}{\partial x}\right) ds.
\end{equation}
In order to evaluate this integral, we introduce new variables
\begin{equation}
\label{eq:variables_z_tilde}
    \widetilde{z}_j(x,s) = \frac{s(s+p)(s+p+q)}{z_j(x,s)}+s,\qquad j=1,2,3.
\end{equation}
Since each $z_j(x,s)$ is a solution of $A_s(z)=x$, it follows that
$\widetilde{z}_j(x,s)$ for $j=1,2,3$ is a solution of the equation
$B(z,s) = x$, where
\begin{equation} \label{Symbol_B}
    B(z,s) = \frac{z(z+p)(z+p+q)}{z-s};
\end{equation}
see \eqref{B(z,s)}.

Taking partial derivatives with respect to $s$ and $x$ on both sides
of $B(\widetilde{z}_j(x,s),s)=x$, and applying the chain rule, we
obtain
\begin{equation} \label{eq:Bpartials}
\begin{aligned}
    \left(\frac{\partial B}{\partial z}(\widetilde{z}_j(x,s),s) \right) \,
    \frac{\partial \widetilde{z}_j(x,s)}{\partial s} +
    \frac{\partial B}{\partial s}(\widetilde{z}_j(x,s),s) & = 0, \\
    \left(\frac{\partial B}{\partial z}(\widetilde{z}_j(x,s),s)\right)
    \frac{\partial \widetilde{z}_j(x,s)}{\partial x} & = 1
\end{aligned}
\end{equation}
for $j=1,2,3$. From \eqref{Symbol_B}, it is elementary to deduce that
\[ \frac{\partial B}{\partial s}(\widetilde{z}_j(x,s),s)
=\frac{\widetilde z_j(x,s)(\widetilde z_j(x,s)+p)(\widetilde
z_j(x,s)+p+q)} {(\widetilde z_j(x,s)-s)^2}=\frac{x}{\widetilde
z_j(x,s)-s}. \] Combining this with \eqref{eq:Bpartials}, we obtain
\begin{equation}\label{eq:derivative_z_i first}
\frac{x}{\widetilde{z}_j(x,s)-s}\frac{\partial \widetilde{z}_j(x,s)}{\partial x}
=-\frac{\partial \widetilde{z}_j(x,s)} {\partial s}, \qquad j=1,2,3.
\end{equation}
By \eqref{eq:variables_z_tilde}, we also note that
\begin{equation}\label{partial x of zj}
\frac{\partial \widetilde z_j(x,s)}{\partial x}
=-\frac{s(s+p)(s+p+q)}{z_j^2(x,s)}\frac{\partial z_j(x,s)}{\partial
x} =-\frac{\widetilde z_j(x,s)-s}{ z_j(x,s)} \frac{\partial
z_j(x,s)}{\partial s}.
\end{equation}
Substituting \eqref{partial x of zj} into \eqref{eq:derivative_z_i
first} gives
\begin{equation}\label{eq:derivative_z_i}
\frac{1}{z_j(x,s)}\frac{\partial z_j(x,s)}{\partial s}
=\frac{1}{x}\frac{\partial \widetilde{z}_j(x,s)} {\partial s},
\qquad j=1,2,3.
\end{equation}

Hence, using \eqref{eq:derivative_z_i} in \eqref{integral_derivative_A} we get
\begin{equation*}
    V_1'(x)=\frac{1}{x}\int_0^{s^*(x)} \left(
    \frac{\partial \widetilde{z}_1(x,s)}{\partial s}-
    \frac{\partial \widetilde{z}_2(x,s)}{\partial s}\right) ds.
\end{equation*}
It then follows from the fundamental theorem of calculus that
\begin{equation} \label{eq:Vprime3}
V_1'(x) = \frac{1}{x}[\widetilde z_1(x,s^*(x))-\widetilde z_2(x,s^*(x))-\lim_{s\to 0+}
(\widetilde z_1(x,s)-\widetilde z_2(x,s))].
\end{equation}

By definition, $s^*(x)$ is the smallest value of $s \geq 0$ for
which $x\in \Gamma_1(s)$. Then $x=\gamma(s^*(x))$ if $x_0<x$ and
$x=\beta(s^*(x))$ if $0<x<x_0$. We can observe from Figure
\ref{Plot_symbol} that $z_1(\gamma(s),s)=z_2(\gamma(s),s)$ and
$z_1(\beta(s),s)=z_2(\beta(s),s)$. Therefore,
 \[ \widetilde z_1(x,s^*(x)) =\widetilde z_2(x,s^*(x)) \]
 and \eqref{eq:Vprime3} reduces to
\begin{equation} \label{eq:Vprime4}
V_1'(x) = -\frac{1}{x}\lim_{s\to 0+}
(\widetilde z_1(x,s)-\widetilde z_2(x,s)).
\end{equation}
From \eqref{eq:variables_z_tilde} and
Lemma~\ref{asymptotic_solutions_alpha}, we find that
\begin{align*}
    \lim_{s\rightarrow 0+} \widetilde z_1(x,s) = 0
    \qquad \text{and} \qquad
    \lim_{s\rightarrow 0+} \widetilde z_2(x,s) =
    -\frac{2(p^2+pq-x)}{2p+q+\sqrt{q^2+4x}}.
\end{align*}
Then \eqref{eq:Vprime4} leads to
\begin{equation} \label{eq:Vprime5}
V_1'(x)=-\frac{1}{x}\frac{2(p^2+pq-x)}{2p+q+\sqrt{q^2+4x}}.
\end{equation}

We obtain $V_1(x)$ by integrating \eqref{eq:Vprime5} with respect to $x$, which
gives
\[ V_1(x)=\sqrt{q^2+4x}-p\log(4x)-q\log(\sqrt{q^2+4x}+q)+C. \]
The constant of integration $C$ can be determined by requiring
$V_1(x_0) = 0$; see \eqref{eq:Vinx0}. This gives
\[ C=-2p-q+p\log(4p^2+4pq)+q\log(2p+2q), \]
and \eqref{external_field} is proved. \qed

\paragraph{Proof of \eqref{external_field V2}.}
Let $x < 0$. Again, it follows from Proposition
\ref{Proposition_Gamma_1_2_alpha} that there exists a unique $s^*(x)
\geq 0$ so that for all $s > 0$,
\begin{equation}\label{def of s star 2}
 x \in \Gamma_2(s)
\quad \Longleftrightarrow \quad s \geq s^*(x).
\end{equation} Then
$\log |z_2(x,s)/z_3(x,s)|=0$ for all $s\geq s^*(x)$, and so by
\eqref{definition-of-V}
\begin{equation} \label{eq:finiteintegralV}
V_2(x)=\int_0^{s^*(x)}  \log \left| \frac{z_2(x,s)}{z_3(x,s)}
\right| ds.
\end{equation}
There also exists a special value
\[ \widetilde x_0 = -q^2/4 = \lim_{s \to 0+} \eta(s) \]
so that $x\in\Gamma_2(s)$ for any $s > 0$ if $x \leq \widetilde
x_0$. Then, for any $x \leq \widetilde x_0$, we have
$s^*(x) = 0$
and
\begin{equation} \label{eq:V2inx0}
V_2(x) = 0.
\end{equation}

If $-q^2/4=\widetilde x_0 \leq x<0$, by the same argument in the
proof of \eqref{external_field}, we have
\begin{equation} \label{eq:V2prime}
V_2'(x) = -\frac{1}{x}\lim_{s\to 0+} (\widetilde z_2(x,s)-\widetilde
z_3(x,s)),
\end{equation}
where $\widetilde z_j(x,s)$, $j=2,3$ are given in
\eqref{eq:variables_z_tilde}. From
Lemma~\ref{asymptotic_solutions_alpha}, we see
\begin{align*}
    \lim_{s\rightarrow 0+} \widetilde z_2(x,s) =
    -\frac{2(p^2+pq-x)}{2p+q+\sqrt{q^2+4x}}
    \qquad \text{and} \qquad
    \lim_{s\rightarrow 0+} \widetilde z_3(x,s) =
    -\frac{2(p^2+pq-x)}{2p+q-\sqrt{q^2+4x}}.
\end{align*}
Hence, \eqref{eq:V2prime} leads to
\begin{equation} \label{eq:V2prime2}
V_2'(x)=-\frac{\sqrt{q^2+4x}}{x}.
\end{equation}

By integrating \eqref{eq:V2prime2} with respect to $x$, we obtain
\[ V_2(x)=-2\sqrt{q^2+4x}+
q\log\Big(\frac{q+\sqrt{q^2+4x}}{q-\sqrt{q^2+4x}}\Big)+ \widetilde
C,\] for $-q^2/4\leq x<0$. The constant of integration $\widetilde
C$ can be determined by requiring $V_2(\widetilde x_0) = 0$; see
\eqref{eq:V2inx0}. This leads to $\widetilde C=0$ and
\eqref{external_field V2} is proved. \qed

\paragraph{Proofs of \eqref{eq:density_nu1_formula}
and \eqref{eq:density_nu2_formula}.}

By the definitions of $\nu_1^\xi$ and $\nu_2^\xi$ in
 \eqref{nu1xi-as-integral} and
\eqref{nu2xi-as-integral}, we obtain from
\eqref{measure_mu_for_recurrence} and \eqref{eq:densitymu2} that
\begin{align*}
\frac{d\nu_1^\xi}{dx}(x)
%=\frac1\xi \int_{0}^{\xi} \frac{d\mu_1^s}{dx}(x) ds
=\frac1{2\pi i \xi} \int_{0}^{\xi} \left(
\frac{1}{z_{1-}(x,s)}\frac{\partial z_{1-}(x,s)}{\partial x}-
\frac{1}{z_{1+}(x,s)}\frac{\partial z_{1+}(x,s)}{\partial x} \right)ds, \\
\frac{d\nu_2^\xi}{dx}(x)
%=\frac1\xi
%\int_{0}^{\xi}\frac{d\mu_2^s}{dx}(x) ds
= \frac1{2\pi i \xi} \int_{0}^{\xi}\left(
\frac{1}{z_{2+}(x,s)}\frac{\partial z_{2+}(x,s)}{\partial
x}-\frac{1}{z_{2-}(x,s)}\frac{\partial z_{2-}(x,s)}{\partial x}
\right)ds.
\end{align*}
With $s^*(x)$ given in \eqref{def of s star 1} and \eqref{def of s
star 2}, it is readily seen that
\begin{align*}
\frac{d\nu_1^\xi}{dx}(x) =\frac1{2\pi i \xi} \int_{s^*(x)}^{\xi}
\left( \frac{1}{z_{1-}(x,s)}\frac{\partial z_{1-}(x,s)}{\partial x}-
\frac{1}{z_{1+}(x,s)}\frac{\partial z_{1+}(x,s)}{\partial x} \right)ds, \\
\frac{d\nu_2^\xi}{dx}(x) = \frac1{2\pi i \xi}
\int_{s^*(x)}^{\xi}\left( \frac{1}{z_{2+}(x,s)}\frac{\partial
z_{2+}(x,s)}{\partial x}-\frac{1}{z_{2-}(x,s)}\frac{\partial
z_{2-}(x,s)}{\partial x} \right)ds.
\end{align*}
Now, by introducing the change of variable
\eqref{eq:variables_z_tilde}, one can evaluate the above integrals
using similar methods as given in the proofs of
\eqref{external_field} and \eqref{external_field V2}. We omit the
details and this completes the proof of Theorem
\ref{theorem:external_field}. \qed

\section*{Acknowledgements}
\noindent The authors are supported by the Belgian Interuniversity
Attraction Pole P06/02. The first author is supported by
FWO-Flanders project G.0427.09. The second author is supported by K.
U. Leuven research grant OT/08/33.

%----------------------------------------------------------------------------

\end{document}